\documentclass[leqno,12pt]{amsart} 

\usepackage{amssymb}

\newtheorem{thm}{Theorem}[section]
\newtheorem{prop}[thm]{Proposition}
\newtheorem{lemma}[thm]{Lemma}

\theoremstyle{definition}
\newtheorem{defn}[thm]{Definition}

\theoremstyle{remark}

\numberwithin{equation}{section}

\def\C{\mathbb{C}}
\def\Z{\mathbb{Z}}
\def\R{\mathbb{R}}

\def\Z{\mathbb{Z}}
\def\E{\mathcal{E}}

\def\BB{\mathbb{B}}
\def\SS{\mathbb{S}}
\def\d{{\text{d}}}
\def\g{\mathfrak{g}}

\def\D{\mathcal{D}}

\def\CC{\mathcal{C}}

\newcommand{\de}{\partial}
\newcommand{\db}{\overline{\partial}}
\newcommand{\ddb}{{\partial }\overline{\partial}}
\newcommand{\pkk}{$p-$K\"ahler }

\newcommand{\KK}{K\"ahler }

\def\B{\mathcal{B}}
\def\G{\mathcal{G}}

\def\W{\mathcal{W}}
\def\A{\mathcal{A}}
\def\G{\mathcal{G}}

\def\K{\mathcal{K}}
\def\S{\mathcal{S}}
\def\PP{\mathcal{P}}
\def\M{\mathcal{M}} 

\begin{document}

\title[weak forms]{Weak forms of $\ddb-$Lemma on compact complex manifolds}

\author{Lucia Alessandrini}
\address{ Dipartimento di Scienze Matematiche Fisiche e Informatiche\newline
Universit\`a di Parma\newline
Parco Area delle Scienze 53/A\newline
I-43124 Parma
 Italy} \email{lucia.alessandrini@unipr.it}



\begin{abstract}
This paper is devoted to give a complete unified study of several weak forms of $\ddb-$Lemma on compact complex manifolds. 
\end{abstract}

\maketitle

\section{Introduction}

One of the relevant properties of compact \KK  manifolds is the so called {\it $\ddb-$Lemma}, which assures that, for every couple of indices $(p,q)$ and for every $d-$closed $(p,q)-$form $\alpha$, the various exactness conditions are equivalent (i.e. $\alpha$ is $d-$exact if and only if it is $\de-$exact, or 
$\db-$exact, or $\ddb-$exact, or also the $(p,q)-$component of a boundary: see Section 4). 
In this way, all important cohomologies are linked: De Rham, Dolbeault, Aeppli, Bott-Chern.

Compact manifolds on which the $\ddb-$Lemma holds have been also called {\it cohomologically K\"ahler manifolds}; notice that there is a  wide class of 
cohomologically K\"ahler compact manifolds, namely class $\CC$ of Fujiki.

An important consequence of the $\ddb-$Lemma is the possibility to find a system of common representatives in the various cohomology classes, that allow to compute cohomology groups (due to compactness, cohomology is finite dimensional). To be precise, while Fr\"ohlicher relations always hold on the manifold $M$, i.e. for every $k$,
$$ dim \ H_{DR}^{k}(M, \C) := b_k(M)  \leq \sum _{p+q=k} h^{p,q}_{\db}(M)  := \sum _{p+q=k} dim \ H^{p,q}_{\db}(M, \C),$$
not always $M$ has a Hodge decomposition (i.e. in the previous formula, equalities hold) or a strong Hodge decomposition, i.e. for every $k,p,q$,
$$b_k(M)  = \sum _{p+q=k} h^{p,q}_{\db}(M), \ \ \  h^{p,q}_{\db}(M) = h^{q,p}_{\db}(M).$$
But when $M$ is a $\ddb-$manifold, i.e. a compact complex manifold on which the $\ddb-$Lemma holds, then $M$ has a strong Hodge decomposition; this implies that for every $p,q$, the Fr\"ohlicher Spectral Sequence degenerates at $E_1$, and there are natural isomorphisms (see Section 2 for the definitions) 
$$H^{p,q}_{\d} \simeq H^{p,q}_{\db} \simeq H^{p,q}_{\de} \simeq H^{p,q}_{\ddb} \simeq H^{p,q}_{\de +\db}.$$

From this kind of results one gets nice geometric properties of the manifold, in particular on deformations of the complex structure (see Section 7).
\medskip

In the last years, some authors studied many weak forms of  the $\ddb-$Lemma to achieve similar geometric properties on compact non \KK manifolds, in particular about deformations of the complex structure of the manifold.

This paper is devoted to  a  unified study of several weak forms of $\ddb-$Lemma on compact complex manifolds, in order to clarify them through the general setting, to emphasize  links and  possible generalizations; we also pay particular attention to the link with generalized \pkk structures, which can also be considered as weak forms of the \KK condition. As a matter of fact, our investigation was motivated by  the following question: for a fixed index p, which conditions, weaker than the $\ddb-$Lemma, assures the equivalence of the various generalized \pkk structures? The answer is given in Theorem 6.7.

Our investigation starts from two (old) papers which are basic in the subject, i.e.  \cite{DGMS} (1976) and \cite{Va} (1986).
\medskip

The plane of the paper is as follows:

1. Introduction

2. The setting in the paper of Varouchas \cite{Va}

3. The setting in the paper of Deligne, Griffiths, Morgan, Sullivan \cite{DGMS}

4. Weak $\ddb-$conditions

5. Generalized \pkk manifolds

6. Cones

7. Some results.

\bigskip

 \section{The setting in the paper of Varouchas \cite{Va}}
 
In \cite{Va}, the author calls \lq\lq regular manifolds\rq\rq those compact complex manifolds where $Ker \ddb = Ker \de + Im \db$ on differential forms (see  forthcoming Definition 2.6). He proves that their cohomologies (De Rham, Dolbeault, Aeppli, Bott-Chern) admit common representatives, using the $\ddb-$elliptic complex (see p. 236 ibidem).
From this paper we toke the  exact sequences  (3.1) and (3.2) which have been used for the first time in \cite{A1}; let us recall them here, and let us collect the first results (which are more or less well known in the literature, see \cite{A}, \cite{AT}, \cite{Po1}, \cite{Po2}, \cite{PU1}, \cite{RWZ1}, \cite{RWZ2}, \cite{RZ}, and so on).
 \medskip
 
  Let $M$ be a compact complex manifold of complex dimension $n$; for every $p$ and $q$, $0 \leq p,q \leq n$, let ${\E}^{p,q}(M)$ denote the Fr\'echet space of complex-valued $(p,q)-$forms.
 Let us consider the following operators:
 $$\de^{p,q} := \de : {\E}^{p,q}(M) \to {\E}^{p+1,q}(M), \ \ \db^{p,q} := \db : {\E}^{p,q}(M) \to {\E}^{p,q+1}(M)$$
 $$\ddb^{p,q} := \ddb : {\E}^{p,q}(M) \to {\E}^{p+1,q+1}(M).$$
 The corresponding cohomology groups are (we use here an old - but much more intuitive - notation, see f.i. \cite{DP}):
 $$H^{p,q}_{\de} = H^{p,q}_{\de}(M, \C) := \frac{Ker \de^{p,q}}{Im \de^{p-1,q}}, \ \ H^{p,q}_{\db} = H^{p,q}_{\db}(M, \C) := \frac{Ker \db^{p,q}}{Im \db^{p,q-1}}$$
 $$H^{p,q}_{\ddb} = H^{p,q}_{\ddb}(M, \C) := \frac{Ker \de^{p,q} \cap Ker \db^{p,q}}{Im \ddb^{p-1,q-1}} = H^{p,q}_{BC}(M, \C), $$
 $$H^{p,q}_{\de + \db} = H^{p,q}_{\de + \db}(M, \C) := \frac{Ker \ddb^{p,q}}{Im \de^{p-1,q} + Im \db^{p,q-1}} = H^{p,q}_{A}(M, \C), $$
 
 When $p=q$, we also consider  real cohomology (here the natural operator is $i \ddb$): 
 $$H_{\ddb}^{k,k}(M, \R) = H_{BC}^{k,k}(M, \R) :=\frac{\{ \varphi \in {\E}^{k,k}(M)_\R;
d\varphi =0\}}{\{\varphi = i\partial\overline{\partial}\alpha ;\alpha \in {\E}^{k-1,k-1}(M)_\R\}}$$
$$H_{\de + \db}^{k,k}(M, \R) = H_{A}^{k,k}(M, \R) :=\frac{\{ \varphi \in {\E}^{k,k}(M)_\R;
i\ddb\varphi =0\}}{\{\varphi = \de \overline\alpha + \db \alpha ; \alpha \in {\E}^{k,k-1}(M)\}}$$
$$H_{d} ^{k,k}(M, \R) :=\frac{\{ \varphi \in {\E}^{k,k}(M)_\R;
d\varphi =0\}}{\{  \varphi \in {\E}^{k,k}(M)_\R; \varphi = d\eta; \eta \in {\E}^{2k-1}(M)_\R\}}$$
$$H_{DR}^{j}(M, \R) :=\frac{\{ \zeta \in {\E}^{j}(M)_\R;
d\zeta =0\}}{\{  \zeta \in {\E}^{j}(M)_\R; \zeta = d\eta; \eta \in {\E}^{j-1}(M)_\R\}} .$$
\medskip

For every $p,q,$ with $ 0 \leq p,q \leq n$, we shall consider also the following quotient spaces:

$$A^{p,q}= A^{p,q}(M, \C) := \frac{Im \de^{p-1,q} \cap Im \db^{p,q-1}}{Im \ddb^{p-1,q-1}}, \ \ F^{p,q} = F^{p,q}(M, \C) := \frac{Ker \ddb^{p,q}}{Ker \de^{p,q} + Ker \db^{p,q}}, $$
$$B^{p,q}= B^{p,q}(M, \C) := \frac{Im \de^{p-1,q} \cap Ker \db^{p,q}}{Im \ddb^{p-1,q-1}}, \ \ D^{p,q} = D^{p,q}(M, \C) := \frac{Ker \de^{p,q} \cap Im \db^{p,q-1}}{Im \ddb^{p-1,q-1}}, $$
$$C^{p,q}= C^{p,q}(M, \C) := \frac{Ker \ddb^{p,q}}{Im \de^{p-1,q} + Ker \db^{p,q}}, \ \ E^{p,q} = E^{p,q}(M, \C) := \frac{Ker \ddb^{p,q}}{Ker \de^{p,q} + Im \db^{p,q-1}}.$$
\medskip

Due to the compactness of $M$, all these spaces are finite dimensional, and their dimension is denoted by the corresponding small letter, f.i. 
$h^{p,q}_{\de + \db} := dim_{\C} H^{p,q}_{\de + \db}$.
\medskip

These spaces are linked by the following exact sequences, where all maps are induced by the identity:
  \begin{equation}\label{1}
0 \to A^{p,q}  \to B^{p,q}  \to H^{p,q}_{\db}  \stackrel{g_{\db}}{\to}  H^{p,q}_{\de + \db}  \to C^{p,q} \to 0
  \end{equation}
  \begin{equation}\label{2}
0 \to D^{p,q}  \to H^{p,q}_{\ddb} \stackrel{f_{\db}}{\to}  H^{p,q}_{\db}  \to E^{p,q}  \to F^{p,q} \to 0
  \end{equation}
  
By means of the isomorphism induced by the conjugation, it is easy to see that:
$$H^{p,q}_{\de + \db} \simeq H^{q,p}_{\de + \db}, \ \ H^{p,q}_{\ddb} \simeq H^{q,p}_{\ddb}, \ \ A^{p,q} \simeq A^{q,p}, \ \ F^{p,q} \simeq F^{q,p}$$
and
$$H^{p,q}_{\db} \simeq H^{q,p}_{\de}, \ \ B^{p,q} \simeq D^{q,p}, \ \ C^{p,q} \simeq E^{q,p},$$
so that we have also the conjugate exact sequences:
  \begin{equation}\label{3}
0 \to A^{p,q}  \to D^{p,q}  \to H^{p,q}_{\de}  \stackrel{g_{\de}}{\to}  H^{p,q}_{\de + \db}  \to E^{p,q} \to 0
  \end{equation}
  \begin{equation}\label{4}
0 \to B^{p,q}  \to H^{p,q}_{\ddb} \stackrel{f_{\de}}{\to}  H^{p,q}_{\de}  \to C^{p,q}  \to F^{p,q} \to 0.
  \end{equation}
\medskip
  
  {\bf Remark 2.1.1} We may also define:
$$\tilde B^{p,q} := B^{p,q} / A^{p,q} \simeq \frac{Im \de^{p-1,q} \cap Ker \db^{p,q}}{Im \de^{p-1,q} \cap Im \db^{p,q-1}}, \ \  \tilde E^{p,q} := E^{p,q} / F^{p,q} \simeq \frac{Ker \de^{p,q} + Ker \db^{p,q}}{Ker \de^{p,q} + Im \db^{p,q-1}},$$
so that sequences (2.1) and (2.2) becomes:
  \begin{equation}\label{5}
0 \to \tilde B^{p,q}  \to H^{p,q}_{\db}  \stackrel{g_{\db}}{\to}  H^{p,q}_{\de + \db}  \to C^{p,q} \to 0
  \end{equation}
  \begin{equation}\label{6}
0 \to D^{p,q}  \to H^{p,q}_{\ddb} \stackrel{f_{\db}}{\to}  H^{p,q}_{\db}  \to \tilde E^{p,q} \to 0,
  \end{equation}
and also by conjugation:
$$\tilde D^{p,q} := D^{p,q} / A^{p,q} \simeq \frac{Ker \de^{p,q} \cap Im \db^{p,q-1}}{Im \de^{p-1,q} \cap Im \db^{p,q-1}}, \ \  \tilde C^{p,q} := C^{p,q} / F^{p,q} \simeq \frac{Ker \de^{p,q} + Ker \db^{p,q}}{Im \de^{p-1,q} + Ker \db^{p,q}},$$
and the corresponding sequences:
  \begin{equation}\label{7}
0 \to \tilde D^{p,q}  \to H^{p,q}_{\de}  \stackrel{g_{\de}}{\to}  H^{p,q}_{\de + \db}  \to E^{p,q} \to 0
  \end{equation}
  \begin{equation}\label{8}
0 \to B^{p,q}  \to H^{p,q}_{\ddb} \stackrel{f_{\de}}{\to}  H^{p,q}_{\de}  \to \tilde C^{p,q} \to 0.
  \end{equation}
\medskip

\begin{lemma} The vector spaces $A^{p,q}, B^{p,q}, C^{p,q}, D^{p,q}, E^{p,q}, F^{p,q}$ have been selected to produce the following useful facts (easy to check), which give a connection among the previous exact sequences:
\begin{enumerate}
\item $\forall \ p,q, \ 0 \leq p,q \leq n,$ the operator $\db$ gives the isomorphism $C^{p,q}  \stackrel{\db}{\to}  D^{p,q+1};$
\item $\forall \ p,q, \ 0 \leq p,q \leq n,$ the operator $\de$ gives the isomorphism $E^{p,q}  \stackrel{\de}{\to}  B^{p+1,q};$
\item $A^{p,q}$ and $F^{p,q}$ are  \lq\lq natural\rq\rq kernels and cokernels of maps involving $H^{p,q}_{\db}$ and $H^{p,q}_{\de}$ (compare (2.5), (2.6), (2.7)  and (2.8)).
\end{enumerate}
\end{lemma}

\begin{prop} We have the following isomorphisms:
\begin{enumerate}
\item $H^{p,q}_{\db} \simeq H^{n-p,n-q}_{\db}, \ \ H^{p,q}_{\de} \simeq H^{n-p,n-q}_{\de};$
 
\item $H^{p,q}_{\ddb} \simeq H^{n-p,n-q}_{\de + \db};$

\item $D^{p,q}  \simeq  C^{n-p,n-q}, \ \ B^{p,q}  \simeq  E^{n-p,n-q}, \ \ A^{p,q}  \simeq  F^{n-p,n-q};$

\item $B^{p+1,q}  \simeq  B^{n-p,n-q}, \ \  D^{p,q+1}  \simeq  D^{n-p,n-q}, \ \ E^{p-1,q}  \simeq  E^{n-p,n-q}, \ \ C^{p,q-1}  \simeq  C^{n-p,n-q}.$
\end{enumerate}
\end{prop}

\begin{proof} Statement (1) is well-known, and also (2) is known: there is a classical non-degenerate pairing 
$H^{p,q}_{\ddb} \times H^{n-p,n-q}_{\de + \db} \to \C$   given by $([\alpha],[\beta]) = \int_M \alpha \wedge \beta$, which gives the isomorphisms
$(H^{p,q}_{\ddb})^* \simeq H^{n-p,n-q}_{\de + \db}$ on compact manifolds.

To prove (3), let us consider the injective map induced by the identity $D^{p,q} \to H^{p,q}_{\ddb}$ given in (2.2) and the surjective map induced by the identity 
$H^{n-p,n-q}_{\de + \db} \to C^{n-p,n-q}$ given in (2.1). As above, we denote by $[\alpha]$ the (opportune) cohomology class of the form $\alpha$.

It is easy to check that the above non degenerate pairing $H^{p,q}_{\ddb} \times H^{n-p,n-q}_{\de + \db} \to \C$
induces a pairing $D^{p,q}\times C^{n-p,n-q} \to \C$, which is non degenerate:
f.i., suppose $[d] \in D^{p,q}$ such that $([d],[c]) = 0 \   \forall \ [c] \in C^{n-p,n-q}$. But $[c]= [w]_{\de + \db}$, with $0 = [c-w] \in C^{n-p,n-q}$, thus since the pairing does not depends on representatives,  $([d],[w]) = 0 \   \forall \ [w] \in H^{n-p,n-q}_{\de + \db}$, which gives $[d]=0$. This proves that $D^{p,q}  \simeq  C^{n-p,n-q}$, and, by conjugation, that $B^{p,q}  \simeq  E^{n-p,n-q}$.

Now, let us consider the injective map induced by the identity $A^{p,q} \to B^{p,q}$ given in (2.1) and the surjective map induced by the identity 
$E^{n-p,n-q} \to F^{n-p,n-q}$ given in (2.2): we can repeat the same considerations as above, to get  $A^{p,q}  \simeq  F^{n-p,n-q}.$

By Lemma 2.1, we get also (4). 
\end{proof}
\medskip

\begin{lemma} 
\begin{enumerate}
\item The following vector spaces are $1-$dimensional: 

$H^{n,n}_{\db}, \  H^{0,0}_{\db}, \ H^{n,n}_{\de}, \  H^{0,0}_{\de}, \ H^{n,n}_{\ddb}, \ H^{n,n}_{\de + \db}, \ H^{0,0}_{\ddb}, \ H^{0,0}_{\de + \db}.$
\item The following vector spaces vanish:

$C^{0,0}, \ E^{0,0}, \ F^{0,0}, \ B^{1,0}, \ D^{0,1}, \ A^{n,n},  \  B^{n,n}, \  D^{n,n}, \ C^{n,n-1}, \ E^{n-1,n},$

and also, for every $p$,

$A^{p,0},  \  D^{p,0}, \  A^{0,p}, \ B^{0,p}, \ F^{p,n}, \ C^{p,n}, \ E^{n,p}, \ F^{n,p}.$
\end{enumerate}
\end{lemma}

\begin{proof} Assertion (1) is well-known; the statement in (2) can be checked by easy computations, using also Proposition 2.2. 
\end{proof}
\medskip

\begin{lemma} For every $p,q$ with $0 \leq p,q \leq n$, we have the following equalities and inequalities:
\begin{enumerate}
\item 
$b^{p,q} = \tilde b^{p,q} + a^{p,q}, \  d^{p,q} = \tilde d^{p,q} + a^{p,q}, \ c^{p,q} = \tilde c^{p,q} + f^{p,q}, \ e^{p,q} = \tilde e^{p,q} + f^{p,q};$ 

\item $b^{p,q} \leq h^{p,q}_{\ddb}, \ d^{p,q} \leq h^{p,q}_{\ddb}, \ c^{p,q} \leq h^{p,q}_{\de + \db}, \ e^{p,q} \leq h^{p,q}_{\de + \db};$

\item 
$h^{p,q}_{\de + \db} + \tilde b^{p,q} = h^{p,q}_{\db} + c^{p,q}, \ h^{p,q}_{\de + \db} + \tilde d^{p,q} = h^{p,q}_{\de} + e^{p,q};$

$h^{p,q}_{\ddb} + \tilde e^{p,q} = h^{p,q}_{\db} + d^{p,q}, \ h^{p,q}_{\ddb} + \tilde c^{p,q} = h^{p,q}_{\de} + b^{p,q};$

\item $h^{p,q}_{\de + \db} + h^{p,q}_{\ddb}  = h^{p,q}_{\de} + h^{p,q}_{\db} + a^{p,q} + f^{p,q};$

\item
$h^{0,q}_{\db} \leq h^{0,q}_{\de + \db}, \ h^{0,q}_{\ddb} \leq h^{0,q}_{\de}, \  h^{p,0}_{\ddb} \leq h^{p,0}_{\db}, \  h^{p,0}_{\de} \leq h^{p,0}_{\de + \db};$

\item
$h^{p,n}_{\de} \leq h^{p,n}_{\ddb}, \ h^{p,n}_{\de + \db} \leq h^{p,n}_{\db}, \  h^{n,q}_{\db} \leq h^{n,q}_{\ddb}, \  h^{n,q}_{\de + \db} \leq h^{n,q}_{\de}.$

\end{enumerate}
\end{lemma}

\begin{proof} Easy computations, using the exact sequences and Lemma 2.3. 
\end{proof}
\medskip

\begin{lemma} For $p+q=1$ or $p+q=2n-1$, we have the following inequalities:
\begin{enumerate}
\item 
$h^{1,0}_{\ddb} \leq  h^{1,0}_{\db} \leq h^{1,0}_{\de + \db}, \ \ h^{1,0}_{\ddb} \leq  h^{1,0}_{\de} \leq h^{1,0}_{\de + \db};$ 

\item $h^{0,1}_{\ddb} \leq  h^{0,1}_{\db} \leq h^{0,1}_{\de + \db}, \ \ h^{0,1}_{\ddb} \leq  h^{0,1}_{\de} \leq h^{0,1}_{\de + \db};$

\item 
$h^{n,n-1}_{\de +\db} \leq  h^{n,n-1}_{\db} \leq h^{n,n-1}_{\ddb}, \ \ h^{n,n-1}_{\de + \db} \leq  h^{n,n-1}_{\de} \leq h^{n,n-1}_{\ddb};$

\item $h^{n-1,n}_{\de + \db} \leq  h^{n-1,n}_{\db} \leq h^{n-1,n}_{\ddb}, \ \ h^{n-1,n}_{\de + \db} \leq  h^{n-1,n}_{\de} \leq h^{n-1,n}_{\ddb}.$
\end{enumerate}
\end{lemma}

\begin{proof} Easy computations, using the previous lemmas. 
\end{proof}

\medskip

Now we can give the definition of regular manifold in the spirit of \cite{Va}: the equivalence of the conditions is given by Lemma 2.1.

\begin{defn} A compact complex manifold is called {\bf regular} if one of the following equivalent conditions hold:
\begin{enumerate}
\item $\forall  p,q,$ it holds: $Ker \ddb^{p,q} \simeq Ker \de^{p,q} + Im \db^{p,q-1}$, i.e. $e^{p,q} =0;$
\item $\forall  p,q,$ it holds: ${Ker \ddb^{p,q}} \simeq {Im \de^{p-1,q} + Ker \db^{p,q}}$, i.e. $c^{p,q} =0;$
\item $\forall  p,q,$ it holds: ${Im \de^{p-1,q} \cap Ker \db^{p,q}} \simeq {Im \ddb^{p-1,q-1}}$, i.e. $b^{p,q} =0;$
\item $\forall  p,q,$ it holds: ${Ker \de^{p,q} \cap Im \db^{p,q-1}} \simeq {Im \ddb^{p-1,q-1}}$, i.e. $d^{p,q} =0.$
\end{enumerate}
\end{defn}

\bigskip

 \section{The setting in the paper of Deligne, Griffiths, Morgan, Sullivan \cite{DGMS}}
 
In this fundamental work, $dd^c-$ (or $\ddb-$)manifolds are studied, to prove the following main result (p. 270, Main Theorem):

\lq\lq Let $M$ be a compact complex manifold, for which the $dd^c-$ (or $\ddb-$)Lemma holds. Then the real homotopy type of $M$ is a formal consequence of the cohomology ring $H^*(M,\R)$.\rq\rq

The $dd^c-$Lemma is given, in the context of a general double complex, in Lemma 5.15 ibidem, as follows:

\lq\lq Let $(K^{*,*}, d', d'')$ be a bounded double complex of vector spaces, and let $(K^*,d)$ be the associated simple complex ($d=d' + d''$). For each integer $k$, $1 \leq k \leq n$, the following conditions are equivalent:

$(a_k)$ \ \ \ in $K^k$, $Ker d' \cap Ker d'' \cap Im d = Im d'd''$

$(b_k)$ \ \ \ in $K^k$, $Ker d'' \cap Im d' = Im d'd''$ \ and \ $Ker d' \cap Im d''  = Im d'd''$

$(c_k)$ \ \ \ in $K^k$, $Ker d' \cap Ker d'' \cap (Im d' + Im d'') = Im d'd''$

$(a^*_{k-1})$ \ \ \ in $K^{k-1}$, $Im d' + Im d'' +Ker d = Ker d'd''$

$(b^*_{k-1})$ \ \ \ in $K^{k-1}$, $Ker d' + Im d''  = Ker d'd''$ \ and \  $Ker d'' + Im d' = Ker d'd''$

$(c^*_{k-1})$ \ \ \ in $K^{k-1}$, $Im d' + Im d'' + (Ker d' \cap Ker d'') = Ker d'd''$. \rq\rq
\bigskip

For a compact complex $n-$dimensional manifold $M$, let us consider the double complex $(\E^{p,q}(M), \de, \db)$ with associated simple complex 
$(\E^k(M) = \oplus_{p+q=k} \E^{p,q}(M), \de + \db = d)$.

In this setting, the authors say that a manifold satisfies the $dd^c-$Lemma, when $(a_k)$ holds for every $k$, and 
satisfies the $\ddb-$Lemma, when $(b_k)$ holds for every $k$.

Comparing with conditions (3) and (4) in Definition 2.6, we get:

\begin{prop} A compact complex manifold is  regular (in the sense of \cite{Va}) if and only if it satisfies the $\ddb-$Lemma (in the sense of \cite{DGMS}).
\end{prop}

To support the careful reader, let us check explicitly the above equivalences in our case $(\E^{p,q}(M), \de, \db)$, since the last one is not straightforward.

{\bf Remarks.} As regards the conditions given above, it holds, for every $k$:

\begin{enumerate}
\item $(a_{k}) \ \iff \ (a^*_{k-1})$

Suppose $(a_{k})$ holds, and let $u = u^{k-1} \in Ker \ddb$; then $v = v^k := du$ belongs to $Ker \de \cap Ker \db \cap Im d = Im \ddb$, thus $du = \ddb w$. This gives 
$$u = (u -  \frac{1}{2} \db w +  \frac{1}{2} \de w)  -  \frac{1}{2} \de w + \frac{1}{2} \db w \in Ker d + Im \de + Im \db,$$
 since $du = d( \frac{1}{2} \db w) - d( \frac{1}{2} \de w)$.
 
 Suppose now $v = v^k \in Ker \de \cap Ker \db \cap Im d$; then $v = du = du^{k-1},$  with $u \in Ker \ddb = Ker d + Im \de + Im \db,$ that is, $u = \de a + \db b + w, \ w \in Ker d$, so that $v = du = \ddb (b-a) \in Im \ddb$.
\medskip
 
 \item $(b_{k}) \ \iff \ (b^*_{k-1})$

Suppose $(b_{k})$ holds, and let $u = u^{k-1} \in Ker \ddb$, so that $\db u \in Ker \de \cap Im \db = Im \ddb$ and $\de u \in Ker \db \cap Im \de = Im \ddb$.  
Hence $\db u = \ddb a, \ \de u = \ddb b$, which gives 
$$u = (u + \de a) -\de a \in Ker \db + Im \de, \ u= (u-\db b) + \db b \in Ker \de + Im \db.$$

 Suppose now $v = v^k \in Im \de \cap Ker \db$ and $w = w^k \in Ker \de \cap Im \db$; then $v = \de u, u = u^{k-1} \in Ker \ddb = Ker \de + Im \db,$ (that is, $u = \db a + r, \ r \in Ker \de$), and $w = \db y, y = y^{k-1} \in Ker \ddb = Im \de + Ker \db,$  (that is, $y = \de b + s, \ s \in Ker \db$), so that $v = \de u = \ddb a \in Im \ddb$,
 $w = \db y = - \ddb b \in Im \ddb$.
\medskip
 
 \item $(c_{k}) \ \iff \ (c^*_{k-1})$

Suppose $(c_{k})$ holds, and let $u = u^{k-1} \in Ker \ddb$; then $v = v^k := \de u$ and  $w = w^k := \db u$ both belong to $Ker \de \cap Ker \db \cap (Im \de + Im \db)$, thus $v = \ddb a, w = \ddb b$. This gives 
$$u = (u -  \db a + \de b) -  \de b +  \db a  \in (Ker \de \cap Ker \db) + Im \de + Im \db.$$
 
 Suppose now $v = v^k \in Ker \de \cap Ker \db \cap (Im \de + Im \db)$; then $v = \de u + \db w$, with $u,w  \in Ker \ddb = (Ker \de \cap Ker \db) + Im \de + Im \db,$ so that $u = \de a + \db a' + r, \ r \in Ker \de \cap Ker \db$, and $w = \de b + \db b' + s, \ s \in Ker \de \cap Ker \db$, hence $v = \de u + \db w = \ddb (a' - b)$.
\medskip
 
 \item $(b_{k}) \ \iff \ (c_{k})$

Let $v = v^{k} = \de u + \db w, \ u,w \in Ker \ddb$, so that $\db w \in Ker \de \cap Im \db = Im \ddb$ and $\de u \in Im \de \cap Ker \db = Im \ddb$.  
Hence $v \in Im \ddb$. The other side is straightforward.
\medskip
 
  \item $(c^*_{k-1}) \ \iff \ (a^*_{k-1})$

The assertion is trivial when $k=1$. 
To prove one side of the statement, notice that $(Ker \de \cap Ker \db) + Im \de + Im \db \subseteq Ker d + Im \de + Im \db.$

On the other hand, let $u = u^{k-1} \in Ker \ddb = Ker d + Im \de + Im \db$; then $u = \de a + \db b + U, \ U = U^{k-1} \in Ker d.$ 

Notice that, for a generic $r-$form $v$ $(r \geq 1)$, the condition : \lq\lq $v \in Ker d$\rq\rq does not imply \lq\lq $v \in Ker \de \cap Ker \db$\rq\rq, but it is equivalent to 
$$\de v^{r,0}=0, \ \ \db v^{r-j,j} + \de v^{r-j-1,j+1}=0, \ \ \db v^{0,r} =0$$
 for $j=0, \dots, r-1$.

As a matter of fact, the operators $\de$ and $\db$ are defined on $\E^{p,q}(M)$, and then extended to $\E^{r}(M)$ by linearity: thus, when $v \in \E^{r}(M),$ 
\lq\lq $v \in Ker \de \cap Ker \db$\rq\rq means that $\de v^{p,q} = 0, \  \db v^{p,q} = 0$ for every $p,q$ with $p+q=r$.

Hence, for every fixed $j=0, \dots, k-2$, let us consider $v = v^k =d(U^{k-1-j,j})$; 
$v \in Ker \de \cap Ker \db \cap Im d$, since $\de v = \ddb U^{k-1-j,j} = \de (- \de U^{k-2-j,j+1}) = 0$, and also $\db v =0$.

But we proved above that $(a_{k}) \ \iff \ (a^*_{k-1})$, so that $v \in Im \ddb$, i.e. 
$\de U^{k-1-j,j} = \ddb r, \  r = r^{k-1-j,j-1}$ and $\db U^{k-1-j,j} = \ddb s, \  s = s^{k-2-j,j}$. 

This gives 
$$\tilde  U^{k-1-j,j} := U^{k-1-j,j} -  \db r +  \de s \in Ker \de \cap Ker \db$$
and 
$$u = \de a + \db b + \Sigma_{j=0}^{k-1} U^{k-1-j,j} = \de \tilde a + \db \tilde b + \Sigma_{j=0}^{k-1} \tilde U^{k-1-j,j}$$ 
as required, where $\tilde a$ contains as summands $a$ and the components of type $s$, and $\tilde b$ contains $b$ and those of type $r$.
\end{enumerate}
\bigskip

  \section{Weak $\ddb-$conditions}

As we said in the introduction, our version of the $\ddb-$Lemma is the following (see also \cite{De}):

\begin{defn} ({\bf $\ddb-$Lemma and $\ddb-$manifold})  We say that a compact complex manifold $M$ is a $\ddb-$ma\-nifold, when  the $\ddb-$Lemma holds for $M$, that is, for every couple of indices $(p,q)$ and for every $d-$closed $(p,q)-$form $u$, the various exactness conditions are equivalent (i.e. $u$ is $d-$exact if and only if it is $\de-$exact, or 
$\db-$exact, or $\ddb-$exact).
\end{defn}

Since a $(p,q)-$form which is $\ddb-$exact, is obviously also $d-$exact, $\de-$exact and $\db-$exact, the $\ddb-$Lemma given in Definition 4.1 is equivalent to say that, for every couple of indices $(p,q)$ and for every  $(p,q)-$form $u \in Ker d = Ker \de \cap Ker \db$, with $u \in Im \de$ (or $u \in Im \db$ or $u \in Im d$), then $u \in Im \ddb$: this is precisely condition $(b_k)$ or $(a_k)$ of \cite{DGMS} for $k=p+q$, as we have seen in Section 3. Hence from now on, we can use several point of view about the $\ddb-$Lemma: the above definition, or the setting in \cite{DGMS}, or the setting in \cite{Va}.
\medskip

Notice that we can complete Definition 4.1 as follows: 
for every couple of indices $(p,q)$ and for every $d-$closed $(p,q)-$form $u$, the following conditions are equivalent:
\begin{enumerate} 
\item $u$ is $d-$exact
\item $u$ is $\de-$exact
\item  $u$ is $\db-$exact
\item $u$ is $\ddb-$exact
\item $u$ is the $(p,q)-$component of a boundary, that is, there are a $(p-1,q)-$form $a$ and a $(p,q-1)-$form $b$ such that $u = \de a + \db b$.
\end{enumerate}

Indeed,  $u = \ddb v = \de (\db v/2) + \db (- \de v/2)$ (and when the form is real, $u = i \ddb v$ for a real form $v$, implies $u = \de (i \db v/2) + \db (-i  \de v/2)$). On the other hand, when $u = \de a + \db b$, by $\db u = 0$ we get that the $d-$closed $(p,q)-$form 
$\de a \in Im \de = Im \ddb$ by the equivalence of the previous conditions, and similarly
$\db b \in Im \db = Im \ddb$, hence $u$ is $\ddb-$exact.

\bigskip
We shall introduce now {\bf weak forms} of the $\ddb-$Lemma, that is, with a fixed couple of indices $(p,q)$, as it was partially done  in some recent papers (see Remark 4.7.1). As a matter of fact, a weak form of  the $\ddb-$Lemma has been proposed for the first time in \cite{FY}, studying deformations of balanced manifolds.
\medskip

But firstly we shall also consider suitable maps, induced by the identity map on cohomology classes, starting from sequences (2.4), (2.2), (2.7), (2.5). For every couple $(p,q)$, we have indeed the following commutative diagram (which is also well known in the literature cited in Section 2):
\[
\begin{array}{ccccc}
\ & \  & H^{p,q}_{\ddb} & \ & \ \\
\ & \swarrow & \  & \searrow & \ \\
H^{p,q}_{\de} & \ & \downarrow & \ & H^{p,q}_{\db} \\
\ & \searrow & \  & \swarrow & \ \\
\ & \  & H^{p,q}_{\de + \db} & \ & \ \\
\end{array}
\]

where, more precisely,
$$f_{\de} := i^{p,q}_{\ddb,\de} : H^{p,q}_{\ddb} \to H^{p,q}_{\de},\ \text{whose Kernel is } B^{p,q}$$
$$f_{\db} := i^{p,q}_{\ddb,\db} : H^{p,q}_{\ddb} \to H^{p,q}_{\db}, \ \text{whose Kernel is } D^{p,q}$$
$$g_{\de} := i^{p,q}_{\de,\de + \db} : H^{p,q}_{\de} \to H^{p,q}_{\de+ \db}, \ \text{whose Kernel is } \tilde D^{p,q}$$
$$g_{\db} := i^{p,q}_{\db,\de + \db} : H^{p,q}_{\db} \to H^{p,q}_{\de + \db}, \ \text{whose Kernel is } \tilde B^{p,q}$$
$$i^{p,q}_{\ddb,\de + \db} : H^{p,q}_{\ddb} \to H^{p,q}_{\de + \db}.$$
\medskip

Notice also that, by Section 2, 
\begin{enumerate}
\item $f_{\de} := i^{p,q}_{\ddb,\de}$ is surjective $\iff  \tilde C^{p,q}=0$
\item $f_{\db} := i^{p,q}_{\ddb,\db}$ is surjective $\iff  \tilde E^{p,q}=0$
\item $g_{\de} := i^{p,q}_{\ddb,\de}$ is surjective $\iff  E^{p,q}=0$
\item $g_{\db} := i^{p,q}_{\ddb,\db}$ is surjective $\iff  C^{p,q}=0.$
\end{enumerate}

\medskip
\begin{defn}
Let $M$ be a complex manifold of dimension $n$, let $0 \leq p,q \leq n$.
We say that $M$ is a  {\bf $(p,q)-$mild $\ddb-$manifold}, or that on $M$ it holds the {\bf $(p,q)-$th mild $\ddb-$Lemma} (condition $\B^{p,q}$ in \cite{ZR}, Definition 1.8, or Definition 3.1  in \cite{RWZ2}, or condition $\BB^{p,q}$ in \cite{RZ}, Notation 3.5), when one of the following equivalent conditions holds:
\begin{enumerate}
\item $B^{p,q} = O$
\item $b^{p,q} = 0$
\item $f_{\de} := i^{p,q}_{\ddb,\de}$ is injective
\item For all $\omega \in \E^{p-1,q}(M)$ with $\ddb \omega = 0$, there is $\alpha \in \E^{p-1,q-1}(M)$ with $\de \omega = \ddb \alpha.$
\end{enumerate}
\end{defn}

Notice that, if $q=0, p \geq 1$, by $B^{p,0} \simeq E^{p-1,0}$, condition (4) means that  there is no  $\omega \in \E^{p-1,0}(M)$ such that $\ddb \omega = 0, \de \omega \neq 0$.

\begin{defn}
Let $M$ be a complex manifold of dimension $n$, let $0 \leq p,q \leq n$.
We say that $M$ is a   {\bf $(p,q)-$dual mild $\ddb-$manifold}, or that on $M$ it holds the {\bf $(p,q)-$th dual mild $\ddb-$Lemma} (see \cite{RWZ2}, Section 3.1), when one of the following equivalent conditions holds:
\begin{enumerate}
\item $D^{p,q} = O$
\item $d^{p,q} = 0$
\item $f_{\db} := i^{p,q}_{\ddb,\db}$ is injective
\item For all $\omega \in \E^{p,q-1}(M)$ with $\ddb \omega = 0$, there is $\alpha \in \E^{p-1,q-1}(M)$ with $\db \omega = \ddb \alpha.$
\end{enumerate}
\end{defn}

Notice that, if $p=0, q \geq 1$, by $D^{0,q} \simeq C^{0,q-1}$, condition (4) means that  there is no  $\omega \in \E^{0,q-1}(M)$ such that $\ddb \omega = 0, \db \omega \neq 0$.

\begin{defn}
Let $M$ be a complex manifold of dimension $n$, let $0 \leq p,q \leq n$.
We say that $M$ is a   {\bf $\widetilde {(p,q)}-$mild $\ddb-$manifold}, or that on $M$ it holds the {\bf $\widetilde {(p,q)}-$th mild $\ddb-$Lemma} (condition $\E^{p,q}$ in \cite{ZR}, Definition 1.8, or condition $\SS^{p,q}$ in \cite{RZ}, Notation 3.5), when one of the following equivalent conditions holds:
\begin{enumerate}
\item $\tilde B^{p,q} = O$
\item $b^{p,q} = a^{p,q}$
\item $g_{\db} := i^{p,q}_{\db,\de + \db}$ is injective
\item For all $\omega \in \E^{p,q}(M)$ with $\db \omega = 0, \omega = \de u$, there is $v \in \E^{p,q-1}(M)$
such that  $\omega = \db v$.
\end{enumerate}
\end{defn}

Notice that, if $q=0, p \geq 1$, condition (4) means that  there is no  $u \in \E^{p-1,0}(M)$ such that $\ddb u = 0, \de u \neq 0$.

\begin{defn}
Let $M$ be a complex manifold of dimension $n$, let $0 \leq p,q \leq n$.
We say that $M$ is a   {\bf $\widetilde {(p,q)}-$dual mild $\ddb-$manifold}, or that on $M$ it holds the {\bf $\widetilde {(p,q)}-$th dual mild $\ddb-$Lemma}, when one of the following equivalent conditions holds:
\begin{enumerate}
\item $\tilde D^{p,q} = O$
\item $d^{p,q} = a^{p,q}$
\item $g_{\de} := i^{p,q}_{\de,\de + \db}$ is injective
\item For all $\omega \in \E^{p,q}(M)$ with $\de \omega = 0, \omega = \db u$, there is $v \in \E^{p-1,q}(M)$
such that  $\omega = \de v$.
\end{enumerate}
\end{defn}

Notice that, if $p=0, q \geq 1$, condition (4) means that  there is no  $u \in \E^{0,q-1}(M)$ such that $\ddb u = 0, \db u \neq 0$.

\begin{defn}
Let $M$ be a complex manifold of dimension $n$, let $0 \leq p,q \leq n$.
We say that $M$ is a   {\bf $(p,q)-$weak $\ddb-$manifold}, or that on $M$ it holds the {\bf $(p,q)-$th weak $\ddb-$Lemma}, when one of the following equivalent conditions holds:
\begin{enumerate}
\item $A^{p,q} = O$
\item $a^{p,q} = 0$
\item For all $\omega \in \E^{p,q-1}(M)$ with $\db \omega = \de \beta$, there is $\alpha \in \E^{p-1,q-1}(M)$ with $\db \omega = \ddb \alpha.$
\end{enumerate}
\end{defn}

\begin{defn}
Let $M$ be a complex manifold of dimension $n$, let $0 \leq p,q \leq n$.
We say that $M$ is a   {\bf $(p,q)-$strong $\ddb-$manifold}, or that on $M$ it holds the {\bf $(p,q)-$th strong $\ddb-$Lemma}, when one of the following equivalent conditions holds:
\begin{enumerate}
\item $B^{p,q} = O, D^{p,q} = O$
\item $b^{p,q} = 0, d^{p,q} = 0$
\item $i^{p,q}_{\ddb,\de + \db}$ is injective
\item For all $\omega \in \E^{p,q}(M)$ with $\omega = \de a + \db b$, there is $\alpha \in \E^{p-1,q-1}(M)$ with 
$\omega = \ddb \alpha$
\item The $(p,q)-$th mild $\ddb-$Lemma and the $(q,p)-$th mild $\ddb-$Lemma hold
\item The $(p,q)-$th dual mild $\ddb-$Lemma and the $(q,p)-$th dual mild $\ddb-$Lemma hold.
\end{enumerate}
\end{defn}
\bigskip

 {\bf Remark 4.7.1}  Let us illustrate here some history about the subject.
 
 \begin{enumerate}
 
 \item The $(n-1,n)-$th weak $\ddb-$Lemma has been proposed in \cite{FY}, Definition 5, in the form of Definition 4.6 (3), to prove Theorem 6. The authors say that this condition is verified when $H^{2,0}_{\db} =0$: of course, $H^{2,0}_{\db} =0$ produces the vanishing of: $H^{n-2,n}_{\db}$ (Proposition 2.2), $E^{n-2,n}$ (Sequence 2.6 and Lemma 2.3), $B^{n-1,n}$ (Lemma 2.1), $A^{n-1,n}$ (Sequence (2.1)).
 
 \item In \cite{AU1}, Proposition 3.7, the authors consider the condition given in Definition 4.3 (3), for the case $(n-1,n)$, that is, the  $(n-1,n)-$th dual mild $\ddb-$Lemma, and remark  in Proposition 3.9 that this condition is not equivalent to that introduced in \cite{FY}, i.e. the $(n-1,n)-$th weak $\ddb-$Lemma.
 
\item More or less at the same time, in \cite{AU2} Definition 2.1, the same authors introduce  the $(n-1,n)-$th strong $\ddb-$Lemma as in Definition 4.7 (3) and (4), and give some results involving it (see Section 7 here). Moreover,  in Theorem 1.1 they use also the condition giving  the  $\widetilde {(n-1,n)}-$th mild $\ddb-$Lemma.

\item In \cite{PU1} the authors define $sGG-$manifolds and characterize them as  $\widetilde {(n,n-1)}-$mild $\ddb-$manifolds in Lemma 1.2.

\item In \cite{RWZ1} we encounter the $(n-1,n)-$th mild $\ddb-$Lemma (Definition 3.1) and the link under the various weak forms of the $\ddb-$Lemma in the case $(n-1,n)$. 

\item Finally, in \cite{RWZ2} (see also \cite{ZR} and \cite{RZ}) the authors introduce in Definition 3.1 the  $(p,q)-$th mild $\ddb-$Lemma as in Definition 4.2 (4), and then also the $(p,p+1)-$th weak $\ddb-$Lemma, the  $(p,q)-$th dual mild $\ddb-$Lemma and the $(p,q)-$th strong $\ddb-$Lemma.

\item In \cite{ZR}, Definition 1.8, one more definition is given, that is: 

\noindent $M \in \D^{p,q} \iff $ $ \forall \omega^{p,q} \in {Im \de^{p-1,q} \cap Ker \db^{p,q}},$ there is $\chi$ such that $\de \chi =0, \ \omega = \db \chi.$

Obviously, this condition lies between conditions $\tilde b^{p,q} = 0$ and $b^{p,q} = 0$.

\item For a fixed $p$, condition $(H_k)$ in \cite{B} corresponds to condition $\tilde b^{p+k,p-k+1} = 0$, while 
condition $(\tilde H_k)$ in \cite{B} corresponds to condition $b^{p+k,p-k+1} = 0$.

 \end{enumerate}
  \medskip
 
  {\bf Remarks 4.7.2.}  
  \begin{enumerate}
  \item Every manifold $M$ is: a  $(p,q)-$mild $\ddb-$manifold, for $(p,q) \in \{(0,q), (1,0), (n,n)\}$; a   $(p,q)-$dual mild $\ddb-$manifold, for $(p,q) \in \{(p,0), (0,1), (n,n)\}$; and so on, by Lemma 2.3. 
  
  Moreover, the cases $(p,q)=(n-1,n)$ or $(n,n-1)$ enjoy particular properties: f.i., the $(n-1,n)-$th mild $\ddb-$Lemma can be characterized by  $h^{n-1,n}_{\ddb}=h^{n-1,n}_{\db}$, and so on.
  
  \item  The fact that $M$ is a  $(p,q)-$mild $\ddb-$manifold implies that $M$ is a  $(p,q)-$weak $\ddb-$manifold and that $M$ is a  $\widetilde {(p,q)}-$mild $\ddb-$manifold, by Lemma 2.4.
  
  \item The fact that $M$ is a  $(p,q)-$mild $\ddb-$manifold implies that the maps 
  $$i^{p-1,q}_{\de,\de + \db}, \ i^{p-1,q}_{\ddb,\db},\  i^{q,p-1}_{\db,\de + \db}, \  i^{q-1,p}_{\ddb,\de}, \  i^{n-p,n-q}_{\de,\de + \db}, \  i^{n-p,n-q}_{\ddb,\db},\  i^{n-q,n-p}_{\db,\de + \db}, \  i^{n-q,n-p}_{\ddb,\de}$$
   are surjective and that the maps 
   $$i^{q,p}_{\ddb,\db}, \ i^{q,p}_{\de,\de + \db}, \  i^{n-p+1,n-q}_{\ddb,\de}, \  i^{n-p+1,n-q}_{\db,\de + \db}, \  i^{n-q,n-p+1}_{\ddb,\db}, \  i^{n-q,n-p+1}_{\de,\de + \db}$$
  are injective, by Lemmas 2.2 and 2.4. Similar remarks hold for the other classes.
  
  \item Notice that, by Lemma 2.4 (2), to get conditions of type (2) in the previous definitions, we only need the vanishing of a cohomology group of type
  $H^{p,q}_{\ddb}$ or $H^{p,q}_{\de + \db}$. 
  \end{enumerate}
  
\bigskip

\section{Generalized \pkk manifolds}

We introduced  $p-$K\"ahler manifolds in \cite{AA} and then in \cite{AB2}, and studied them mainly in the compact case:
$p-$K\"ahler manifolds enclose K\"ahler and balanced manifolds, and seem to be a nice generalization of the K\"ahler setting. Later on, also pluriclosed (SKT) manifolds have been proposed as a good generalization of K\"ahler manifolds, and many others.

Thus a deep investigation of this type of structures (no more metrics, in general) was needed: we proposed in \cite{A1} a general setting, those of 
{\it generalized $p-$K\"ahler manifolds}, which enclose all the known classes of non-K\"ahler manifolds that can be characterized by a strictly weakly positive 
\lq\lq closed\rq\rq form (see f. i.  \cite{A3}). In the last years, some of them have been studied (not with the same name!) by other authors: hence we give in Remark 5.1.1  a sort of  dictionary; moreover, a brief survey of the whole history can be seen in \cite{A3}.
\medskip

 \begin{defn} Let $X$ be a complex manifold of dimension $n \geq 2$, let $p$ be an integer, $1 \leq p \leq n-1$.
 
\begin{enumerate}
\item $X$ is a {\it $p-$K\"ahler (pK) manifold} if it has a closed transverse (i.e. strictly weakly positive) $(p,p)-$form $\Omega \in \E^{p,p}(X)_{\R}$. 

\item $X$ is a {\it weakly $p-$K\"ahler (pWK) manifold} if it has a transverse $(p,p)-$form $\Omega$ with $\de \Omega = \ddb \alpha$ for some form $\alpha$.

\item $X$ is a {\it $p-$symplectic (pS) manifold} if it has a closed transverse  real $2p-$form $\Psi \in \E^{2p}(X)$; that is, $d \Psi = 0$ and $\Omega := \Psi^{p,p}$ (the  $(p,p)-$component of $\Psi$) is transverse.

\item  $X$ is a {\it $p-$pluriclosed  (pPL) manifold} if it has a transverse $(p,p)-$form $\Omega$ with $\ddb \Omega = 0.$
\end{enumerate}

\end{defn}

Notice that:
$pK \Longrightarrow pWK  \Longrightarrow pS  \Longrightarrow pPL.$ 
When $X$ satisfies one of these definitions, it is called a {\bf generalized $p-$K\"ahler manifold}. The form $\Omega$, called a generalized  \pkk form, is said to be \lq\lq closed\rq\rq; moreover, $\Omega > 0$ means that $\Omega$ is transverse. 

\medskip

{\bf 5.1.1 Remark.} $1PL$ corresponds to pluriclosed (\cite{E}) or SKT (\cite{FT}); $1S$ to hermitian symplectic (\cite{ST}), $1K$ to K\"ahler.
Moreover, $(n-1)PL$ manifolds (or metrics) are called standard or Gauduchon; $(n-1)S$ corresponds to strongly Gauduchon (\cite{Po1}, \cite{X}), $(n-1)WK$ manifolds are called superstrong Gauduchon (\cite{PU1}), $(n-1)K$ corresponds to balanced (\cite{Mi}). Last but not least, similar names to $pPL$ and $pS$ manifolds are given in \cite{B}, where the author present them as a new generalization of the old paper \cite{AA} \dots
\medskip

{\bf 5.1.2 Remark.}  In \cite{A1} we noticed that, on a $\ddb-$manifold, for every $p, 1 \leq p \leq n-1$, it is the same to have a pPL structure, or a pS structure, or a pWK structure. Thus, in particular, every compact manifold is strongly Gauduchon (i.e. $(n-1)$S) and superstrong Gauduchon (i.e. $(n-1)$WK).
This was our motivation to investigate weak forms of $\ddb-$Lemma.
\medskip

{\bf 5.1.3 Remark.}  Generalized $p-$K\"ahler manifolds can be studied also using positive currents: see the Main Theorem 3.1 in \cite{A3}.

\bigskip

 \section{Cones}
 
 Let $M$ be a compact complex manifold. To study particular hermitian metrics on $M$, it is often useful to consider cones in suitable cohomology spaces of $M$. For instance, the \KK cone $\K$ of a compact (K\"ahler) manifold is, by definition, the set of cohomology classes of  $(1,1)-$forms associated with \KK metrics. Notice that $\K$ can be considered as an open convex cone in $H_{d} ^{1,1}(M, \R)$ or in $H_{\ddb} ^{1,1}(M, \R)$.

We will study cones of transverse forms (denoted by $\Omega > 0$) in $H_{\ddb} ^{p,p}(M, \R)$ and in $H_{\de + \db} ^{p,p}(M, \R)$, starting from the (Bott-Chern) \KK  cone, the balanced cone and the Gauduchon cone (see \cite{Po2}, \cite{RZ}, \cite{AU2}, \cite{CRS} and others).
 
 \medskip
 \begin{defn} Let $M$ be a compact complex manifold of dimension $n$. 
 
 The {\bf \KK cone} of $M$ is 
 $$\K_M := \{ [\omega] \in H_{\ddb} ^{1,1}(M, \R) / \omega > 0 \};$$
 the {\bf balanced cone} of $M$ is
 $$\B_M := \{ [\Omega] \in H_{\ddb} ^{n-1,n-1}(M, \R) / \Omega > 0 \};$$
 the {\bf Gauduchon cone} of $M$ is
 $$\G_M := \{ [\Omega] \in H_{\de + \db} ^{n-1,n-1}(M, \R) / \Omega > 0 \}.$$ 
 \end{defn}

{\bf Remark 6.1.1} All the previous cones are open convex cones. We have: $\G_M \neq \emptyset$ by the result of Gauduchon (\cite{Ga}), $\K_M \neq \emptyset$
if and only if $M$ is K\"ahler, $\B_M \neq \emptyset$ if and only if $M$ is balanced.
\medskip

We say that a cone {\bf degenerates} if it encloses the whole space, i.e. when every cohomology class contains a transverse form. Since the cones are open, they degenerate if and only if the null class contains a transverse form. This is impossible for $\K_M$, because we would get a smooth function $f$ such that $i \ddb f > 0$ on a compact manifold.

On the other hand, the cones $\B_M$ and $\G_M$ may degenerate, as we can see through the following example.
\medskip

{\bf Example 6.1.2} (see \cite{Y} or \cite{A3}) Take $G=SL(2,\C)$, $\Gamma = SL(2,\Z)$, and  consider the holomorphic $1-$forms $\eta, \alpha, \beta$ on $M := G/ \Gamma$ induced by the standard basis for $\g^*$: it holds
$$ d \alpha = -2 \eta \wedge \alpha, \ \ d \beta = 2 \eta \wedge \beta, \ \ d \eta = \alpha \wedge \beta.$$
The standard fundamental  form, given by $\omega = \frac{i}{2} (\alpha \wedge \overline \alpha + \beta \wedge \overline \beta +\eta \wedge \overline \eta )$, satisfies $d \omega^2=0$, so that $\omega^2$ is a balanced form: but it is exact, since
$$\omega^2 = d(\frac{1}{16} \alpha \wedge d \overline  \alpha + \frac{1}{16} \beta \wedge d \overline  \beta +\frac{1}{4} \eta \wedge d \overline \eta ).$$
\medskip

We proved on this subject a nice characterization result: see  \cite{A3},  Theorem 5.1 and the remarks after the theorem, in particular Remark 5.2.4. In our setting, we have the following results (Theorem 5.1(4) in \cite{A3} gives $(i)$, while Theorem 5.1(1)  gives $(ii)$): 

\begin{prop}

Let $M$ be a {compact} complex manifold of dimension $n \geq 2$. 

(i) The following statements are equivalent:

\begin{enumerate}
\item $\G_M$ degenerates (i.e. $\G_M \simeq H_{\de +\db}^{n-1,n-1}(M, \R)$).

\item $[0] \in \G_M$. 

\item There is a transverse  $(n-1,n-1)-$form $\Omega$ on $M$ which is the component of a boundary (i.e. there is a form $\beta$ such that $\Omega = \de \overline \beta + \db \beta$). 

\item There are no non trivial closed currents $T \in {\E}_{n-1,n-1}'(M)_{\R}$, $T \geq 0$ (in particular, $M$ does not belong to class $\CC$ of Fujiki, where there are \KK currents).

\end{enumerate}

(ii) The following statements are equivalent (see Lemma 4.18 in \cite{RZ}):

\begin{enumerate}
\item $\B_M$ degenerates (i.e. $\B_M \simeq H_{\ddb}^{n-1,n-1}(M, \R)$).

\item $[0] \in \B_M$. 

\item There is a transverse  $(n-1,n-1)-$form $\Omega$ on $M$ which is $\ddb-$exact.

\item There are no non trivial $\ddb-$closed currents $T \in {\E}_{n-1,n-1}'(M)_{\R}$, $T \geq 0$ (in particular, $M$ does not belong to class $\CC$ of Fujiki, where there are \KK currents).

\end{enumerate}
\end{prop}
\medskip

Before to go to the case of a generic index $p, \ 1 \leq p \leq n-1$, we complete the case $p=n-1$, which is the most popular. 

In \cite{Po2}, Popovici studied $sG-$manifolds (in our setting, $(n-1)S$ manifolds) by means of the strongly Gauduchon cone, which is defined as follows. 
\medskip

Let $T^{n-1,n-1} : H_{\de + \db}^{n-1,n-1}(M, \C) \to H_{\db}^{n,n-1}(M, \C)$ given by $T([\Omega]_{\de + \db}) = [\de \Omega]_{\db}$. 

This map is well defined, and can be detected (matching exact sequences (2.3) and (2.1)) as a composition of maps, using the isomorphism $E^{n-1,n-1} \simeq B^{n,n-1}$ (notice also that $C^{n,n-1}$ = 0):
\medskip

$  \to D^{n-1,n-1}  \to H^{n-1,n-1}_{\de}  \stackrel{g_{\de}}{\to}  H^{n-1,n-1}_{\de + \db}  \to E^{n-1,n-1} \to 0$

$\hspace{4.55 cm} 0 \to A^{n,n-1}  \to B^{n,n-1}  \to H^{n,n-1}_{\db}  \stackrel{g_{\db}}{\to}  H^{n,n-1}_{\de + \db}  \to 0$
\medskip

\begin{defn} (see \cite{Po2}) The strongly Gauduchon cone of $M$ (which may be empty) is 
$$\S\G_M:= \G_M \cap Ker T^{n-1,n-1} \subset H_{\de + \db}^{n-1,n-1}(M, \C).$$
\end{defn}
\medskip

To complete the picture, let us consider also the map
$$W^{n-1,n-1} : H_{\de + \db}^{n-1,n-1}(M, \C) \to H_{\ddb}^{n,n-1}(M, \C)$$
 given by $W([\Omega]_{\de + \db}) = [\de \Omega]_{\ddb}$. This map is well defined, 
since $d(\de \Omega)=0, \ W(\de u + \db v) = [\ddb v]_{\ddb} = 0$.

As above, we detect $W$ (matching exact sequences (2.7) and (2.8)) as a composition of maps, using the isomorphism $E^{n-1,n-1} \simeq B^{n,n-1}$ (notice also that $\tilde C^{n,n-1}$ = 0):
\medskip

$ 0  \to \tilde D^{n-1,n-1}  \to H^{n-1,n-1}_{\de}  \stackrel{g_{\de}}{\to}  H^{n-1,n-1}_{\de + \db}  \to E^{n-1,n-1} \to 0$

$\hspace{6.7 cm} 0  \to B^{n,n-1}  \to H^{n,n-1}_{\ddb}  \stackrel{f_{\de}}{\to}  H^{n,n-1}_{\de}  \to 0$
\medskip

\begin{defn}  The weakly Gauduchon cone of $M$ (which may be empty) is 
$$\W\G_M:= \G_M \cap Ker W^{n-1,n-1} \subset H_{\de + \db}^{n-1,n-1}(M, \C).$$
\end{defn}
\medskip

{\bf Remark 6.4.1}
We have $\W\G_M \subseteq \S\G_M$, since $T^{n-1,n-1} = f_{\db}^{n,n-1} \circ  W^{n-1,n-1}$.

\medskip

We got the following situation: $\W\G_M \subseteq \S\G_M \subseteq \G_M$; in the next Proposition we give conditions to assure the equality of the cones. The proof is a particular case ($p = n-1$) of that of the forthcoming Theorem 6.7.

\medskip

\begin{prop} Let $M$ be a compact complex manifold. In the above notation:
\begin{enumerate}
\item $\G_M = \S\G_M$ (i.e. $M$ is a $sGG-$manifold, see \cite{Po2}, \cite{PU1}) $\iff T^{n-1,n-1} = 0 \iff \tilde B^{n,n-1} =0.$
\item (suppose the cones are not empty) $\W\G_M = \S\G_M \iff $

$Ker T^{n-1,n-1} = Ker W^{n-1,n-1} \ \iff A^{n,n-1} =0.$
\item $\G_M = \W\G_M \iff W^{n-1,n-1} = 0 \ \iff  B^{n,n-1} =0.$
\item The (equivalent) conditions in Proposition 6.2(i) imply that $M$ is $(n-1)$WK and $(n-1)$S.
\end{enumerate}
\end{prop}

{\bf Remark 6.5.1} As for $sGG-$manifolds, one can see \cite{PU1} and  \cite{PU2}. In particular, in \cite{PU2}, Theorem 1.4, the authors characterize $sGG-$manifolds by the condition
$h^{0,1}_{\ddb}=h^{0,1}_{\db}$. As a matter of fact, since by Lemma 2.3(2) we have: $c^{n,n-1} =0$, the condition $\tilde b^{n,n-1} = 0$  in Proposition 6.5(1) is equivalent by Lemma 2.4(3) to  condition $h_{\de + \db}^{n,n-1} = h_{\db}^{n,n-1}$, or (by Proposition 2.2) to $h_{\ddb}^{0,1} = h_{\db}^{0,1}$.

\bigskip
For the case of a generic $p, \ 1 \leq p \leq n-1$, we define the opportune maps and the cones as follows.

\begin{defn}  Let $M$ be a compact complex manifold of dimension $n$, let $1 \leq p \leq n-1$.
\begin{enumerate}
\item The {\bf $p-$\KK cone} of $M$ is 
 $$p\K_M := \{ [\omega] \in H_{\ddb} ^{p,p}(M, \R) / \omega > 0 \};$$
 \item the {\bf $p\PP_M$ cone} of $M$ is
 $$p\PP_M := \{ [\Omega]^{p,p}_{\de + \db} \ / \Omega > 0 \} \subseteq H_{\de + \db}^{p,p}(M, \R);$$
 \item the {\bf $p\tilde \S_M$ cone} of $M$ is 
 $p\PP_M \cap Ker T^{p,p}$, where 
 
 $T^{p,p} : H_{\de + \db}^{p,p}(M, \C) \to H_{\db}^{p+1,p}(M, \C)$ is given by $T([\Omega]_{\de + \db}) = [\de \Omega]_{\db}$
\smallskip

 \item the {\bf $p\W_M$ cone} of $M$ is 
 $p\PP_M \cap Ker W^{p,p}$, where 
 
 $W^{p,p} : H_{\de + \db}^{p,p}(M, \C) \to H_{\ddb}^{p+1,p}(M, \C)$ is given by $T([\Omega]_{\de + \db}) = [\de \Omega]_{\ddb}$

\end{enumerate}
\end{defn}
\medskip

{\bf Remarks 6.6.1} \begin{enumerate}
\item Notice that only the $p-$\KK cone lies in $H_{\ddb} ^{p,p}(M, \R)$, while the other cones are contained in $H_{\de + \db}^{p,p}(M, \R).$
\item All cones may be empty, except for $(n-1)\PP_M = \G_M$. As a matter of fact, $(n-1)\PP_M = \G_M$, $(n-1)\tilde \S_M = \S\G_M$, $(n-1)\W_M = \W\G_M$.
\item Like $\G_M$, all the $p\PP_M$ are open convex cones.
\item The cone $1\tilde \S_M$ is studied in \cite{Ma}.
\end{enumerate}
\medskip

So we have the following general result:

\begin{thm} Let $M$ be a compact complex manifold of dimension $n$, let $1 \leq p \leq n-1$.
 In the above notation (suppose also that the cones are not empty):
\begin{enumerate}
\item $p\PP_M = p\tilde \S_M \ \iff T^{p,p} = 0 \iff \tilde B^{p+1,p} =O \iff M$ is a $\widetilde {(p+1,p)}-$mild $\ddb-$manifold.
\item $p\tilde \S_M = p\W_M \iff Ker T^{p,p} = Ker W^{p,p} \ \iff A^{p+1,p} =O \iff M$ is a ${(p+1,p)}-$weak $\ddb-$manifold.
\item $p\PP_M = p\W_M \iff W^{p,p} = 0 \ \iff  B^{p+1,p} =O \iff M$ is a ${(p+1,p)}-$mild $\ddb-$manifold.
\item When the open cone $p\PP_M$ degenerates (i.e. $p\PP_M \simeq H_{\de + \db}^{p,p}(M, \R)$), or, equivalently, when  $[0] \in p\PP_M$, then $M$ is not only pPL, but also pWK and pS.
\end{enumerate}
\end{thm}

\begin{proof}
\begin{enumerate}
\item $p\PP_M$ is a non empty open convex cone in $H_{\de + \db}^{p,p}(M, \R)$, while $Ker T^{p,p}$ is a linear subspace there. Thus (see also Observation 5.2 in \cite{Po2}) it holds $p\PP_M = p\tilde \S_M \iff Ker T^{p,p} = H_{\de + \db}^{p,p}(M, \R).$ Moreover, we have:
$$\tilde B^{p+1,p} =O \iff B^{p+1,p} = A^{p+1,p} \iff Im \de \cap Ker \db = Im \de \cap Im \db \iff T^{p,p}=0.$$
\item $Ker T^{p,p} \subseteq Ker W^{p,p}$ says that for every $\Omega = \Omega^{p,p}$ with $\ddb \Omega =0$, when $[\de \Omega]_{\db} =0$ (i.e. $\de \Omega = \db \chi$), then  $[\de \Omega]_{\ddb} =0$ (i.e. $\de \Omega = \ddb \alpha$); this is precisely the condition  $A^{p+1,p} =O.$
\item The proof goes as for (1).
\item If $[0] \in p\PP_M$, then both $p\W_M:= p\PP_M \cap Ker W^{p,p}$ and $p\tilde \S_M:= p\PP_M \cap Ker T^{p,p}$ cannot be empty. (Compare with Theorem 5.1 in \cite{A3}, where condition 5.1(4) corresponds to our $[0] \in p\PP_M$, while condition 5.1(3) implies 5.1(4) and assures that $M$ is \pkk, i.e. $p\K_M \neq \emptyset$).
\end{enumerate}
\end{proof}

{\bf Remark 6.7.1} Except for the case $p=n-1$, $p\tilde \S_M$ is not the cone of  classes of $p$S forms. For instance, when $p=1$, 
$[\omega] = [\omega^{1,1}] \in 1\tilde \S_M$
means: $\omega > 0, \ddb \omega = 0, \de \omega = \db \eta^{2,0}$, while \lq $\omega$ is a 1S-form\rq \ means that 
 $\omega > 0,  \de \omega = \db \eta^{2,0}, \de \eta^{2,0} = 0,$ i.e. the form $\psi = - \eta^{2,0} + \omega - \overline \eta^{2,0}$ is closed (compare \cite{Ma}).
 
 Denote by  $p \S_M$ the open convex cone of classes of $p$S forms: we have $$p \W_M \subseteq p \S_M \subseteq p\tilde \S_M  \subseteq p \PP_M.$$ 

Hence the condition  $\tilde b^{p+1,p} =0$ does not imply that every $p$PL form is a $p$S form: we need also the vanishing else of 
$A^{p+1,p}$ or of $B^{p+1,p}$. But of course  condition  $b^{p+1,p} =0$  implies in particular that  $p \PP_M = p \S_M$.

Referring to \cite{ZR}, Definition 1.8 (see Remark 4.7.1(7)), one can easily check that when $M \in \D^{p+1,p}$, then $p \PP_M = p \S_M$, because when $\Omega$ is a $pPL-$form, so that $\ddb \Omega=0$, then we get a form $\chi$ such that $\de \Omega = \db \chi$ and $\de \chi =0$, hence $\Psi := \Omega - \chi - \overline \chi$ is a $pS-$form.

Moreover, in \cite{B}, Section 5, the author defines the cones $\A_p(M)$, which corresponds to $p \PP_M$, and $\CC_p(M)$, that is $p \S_M$. Some results about these cones (see Proposition 5.4 ibidem) are connected with our Theorem 6.7; in particular in Proposition 5.4(i) a condition is given (in terms of the vanishing of $\tilde b^{p,q}$ for suitable $p,q$) to assure that $p \PP_M = p \S_M$, while Proposition 5.4(ii) is a particular case of Theorem 6.7(1).
\medskip

{\bf Remark 6.7.2} As in the classical cases (see Remark 6.1.1), $M$ is a $pPL (pS, pWK, pK)$ manifold if and only if the corresponding cone 
 $p \PP_M$ ($p \S_M, p \W_M, p \K_M$) is not empty. But the statement (for instance) \lq\lq $M$ is $pS$ if and only if $M$ is $pWK$\rq\rq does not mean that the cones coincide, but only that, when the cone $p \S_M$ is not empty, then also $p \W_M \neq \emptyset$. This is much more weak than Theorem 6.7.
\medskip

{\bf Remark 6.7.3} As we stated in Theorem 6.7, to get $p\PP_M = p\W_M$ we need $b^{p+1,p} =0$. This condition can be expressed in many ways, using the results of Section 2  (see Remark 4.7.2(3)), f.i. $e^{p,p}=0$, or the injectivity of $i^{p+1,p}_{\ddb,\de} : H^{p+1,p}_{\ddb} \to H^{p+1,p}_{\de}$; moreover, there are also other conditions which assure $p\PP_M = p\W_M$, as the vanishing of 
$h^{p,p}_{\de + \db}$, or $h^{p+1,p}_{\ddb}$, or  $h^{p,p+1}_{\ddb}$
 (see Remark 6.5.1).

\medskip

To end this section, let us consider the case $n=2$.
\begin{prop} Let $M$ be a compact complex surface ($n=2$). The following conditions are equivalent:
\begin{enumerate}
\item $M$ is \KK
\item $h^{2,1}_{\de + \db} = h^{2,1}_{\db}$
\item $\tilde b^{2,1} =0$
\item All  numbers $a^{i,j}, b^{i,j}, c^{i,j}, d^{i,j}, e^{i,j}, f^{i,j}$ vanish.
\end{enumerate}
\end{prop}

\begin{proof}

(1) implies (4) because a \KK manifold is a $\ddb-$manifold.

(4) implies (3): obvious.

(3) implies (1): in fact, by Theorem 6.7, $\tilde b^{n,n-1} =0$ is equivalent to $(n-1)\tilde \S_M  = (n-1)\PP_M,$ that is, 
$\S\G_M = \G_M$, so that $\S\G_M$ cannot be empty, i.e. $M$ is 1S. This assures that $M$ is \KK (see f.i. \cite{A2}, Remark 2.8). 

Finally, by Lemma 2.4, (3) is equivalent to (2).
\end{proof}

 \bigskip
 
  \section{Some results}

\medskip

We collect in this last section a miscellanea of results, giving only references for the proofs, and some indications to find examples of manifolds satisfying special forms of $\ddb-$Lemma. The focus is on deformations of complex structures.

\medskip

{\bf a) From local to global.} About this topic, we recall only a couple of results of \cite{AU1} (see also \cite{RWZ1} Section 3 for a comment), i.e.  

\lq\lq On a $(n-1,n)-$th dual mild manifold, any locally conformal balanced structure is also globally conformal balanced\rq\rq (Theorem 3.5), and

 \lq\lq On a $(n-1,n)-$th dual mild manifold, any locally conformal \KK structure is also globally conformal \KK\rq\rq (Proposition 3.7).
 \medskip

{\bf b) Cones and \pkk structures} 

In \cite{A1} we proved that, in the class of $\ddb-$manifolds, every $pPL$ manifold is also $pS$ and $pWK$ for every $p, \ 1 \leq p \leq n-1$; 
using Theorem 6.7, we can be more precise. For every $p, \ 1 \leq p \leq n-1$:
\begin{enumerate}
\item If $M$ is a $\ddb-$manifold, then the following cones coincide:
$$p \W_M = p \S_M = p\tilde \S_M = p \PP_M.$$  
\item $M$ is a $(p+1,p)-$mild $\ddb-$manifold if and only if $p\PP_M = p\tilde \S_M = p \S_M = p\W_M$; thus a $pPL$ manifold, which is a $(p+1,p)-$mild $\ddb-$manifold, is also a $pS$ and a $pWK$ manifold.
\item $M$ is a $(p+1,p)-$weak $\ddb-$manifold if and only if $p\tilde \S_M = p \S_M = p\W_M$; thus a $pS$ manifold, which is a $(p+1,p)-$weak $\ddb-$manifold, is also a $pWK$ manifold.

\end{enumerate}
\medskip

{\bf c) Holomorphic families of compact complex manifolds} 

Recall that a holomorphic family of compact complex manifolds is a proper holomorphic submersion $\pi : \M \to \Delta$ between complex manifolds, where $\Delta$ is assumed to be an open ball of $\C^m$ containing the origin (it suffices to assume $m=1$); all fibres are compact $n-$dimensional manifolds $\pi^{-1}(t) := M_t$, which are diffeomorphic each other (i.e. only the complex structure $J_t$ of $M_t$ varies with $t \in \Delta$; in particular, Betti numbers are invariant).
Following \cite{Po1}, a property (P) is open (closed) under holomorphic deformations if for every holomorphic family and for every $t_0 \in \Delta$:

(open): $M_{t_0}$ has the property (P) implies that $M_t$ has the property (P) for all $t \in \Delta$ close to $t_0$;

(closed): $M_t$ has the property (P) for all $t \in \Delta - t_0$  implies that $M_{t_0}$ has property (P).
\medskip

A lot of work has been done in studying openness or closeness of properties as being K\"ahler, balanced, $\ddb-$manifolds and so on: one can see \cite{Po1}, \cite{Po2}, \cite{PU1}, \cite{X}, \cite{AU2}, \cite{RZ}, \cite{RWZ2}, \cite{B}, \cite{COUV} and some others.

In an old paper (\cite{AB1}), we proved that, while the property of being \KK is open, the property of being \pkk ($p > 1$) is not open. For $p=n-1$, the example is the Iwasawa manifold $I_3$, which is  a $(2,3)-$weak $\ddb-$manifold and a $(2,3)-$dual mild $\ddb-$manifold, but not a $(2,3)-$mild $\ddb-$manifold (by Lemma 2.4 and  \cite{A}, Appendix A) (this example and its deformations will appear more and more).
\medskip

But now we start from  \cite{FY}, where the authors recall that the property of being a $\ddb-$manifold is open, and also the property of being a balanced $\ddb-$manifold is open. By the way, they noticed that the $\ddb-$Lemma is too strong, in fact they prove (see Theorem 6 and Corollary 7):

\lq\lq When $X_t$ is a $(n-1,n)-$weak $\ddb-$manifold for small $t \neq 0$ (f.i., when $h^{2,0}_{\db}(X_t) =0)$, and $X_0$ is balanced, then $X_t$ is balanced for small $t$\rq\rq.

Notice that the request is on every $X_t$; as a matter of fact in \cite{RWZ1} a lot of examples are recalled, and in particular Example 3.13, where $X_0$ and $X_t$ are balanced (they have the real structure of $I_3$),  $X_0$ is a $(2,3)-$mild $\ddb-$manifold, but $X_t$ is not a $(2,3)-$weak $\ddb-$manifold.

In the proof, Fu and Yau use the real diffeomorphism $X_0 \simeq X_t$ to get on $X_t$ a real $d-$closed $(2n-2)-$form $\Omega_t$ corresponding to the balanced $(n-1,n-1)-$form $\Omega_0$ on $X_0$. Its $(n-1,n-1)-$component $\Omega_t^{(n-1,n-1)}$ satisfies 
$\db_t  \Omega_t^{(n-1,n-1)} = - \de_t \Omega_t^{(n-2,n)}$, so that, by the hypothesis, $\db_t  \Omega_t^{(n-1,n-1)} = i \de_t \db_t \Psi_t$; in this manner,
$$\tilde \Omega_t:=   \Omega_t^{(n-1,n-1)}  - i \de_t  \Psi_t - i \db_t \overline \Psi_t$$
is $d-$closed, and it becomes transverse for a small $t$. 
\medskip

When $p < n-1$,  in \cite{RWZ1} (Proposition 4.12 and Remark 4.13) and \cite{RWZ2} (Proposition 1.5) the authors notice that, for a sufficiently small $t$, any smooth real extension of a transverse $(p,p)-$form is still transverse, so that the obstruction relies in the closure (as regards $pS$ manifolds, one can see \cite{B}, Theorem 1.1).

In the first paper, they define the  ${(n,n-1)}-$th mild $\ddb-$Lemma, compare them with the other weak forms of the $\ddb-$Lemma at the level $(n,n-1)$, also with a lot of examples, and then prove the following result (Theorem 3.11): 

\noindent \lq\lq When $X_0$ is a balanced $(n-1,n)-$mild $\ddb-$manifold, then $X_t$ is balanced for small $t$\rq\rq.

On the other hand, $I_3$ shows that the property of being a $(n-1,n)-$mild $\ddb-$manifold is not open, and that the result is not true when mild is replaced by weak.

Moreover in the just cited paper \cite{RWZ1}, on page 7, it is noticed that the property \lq\lq$(n-1,n)-$weak $\ddb-$manifold\rq\rq is not open, where on the contrary  the property \lq\lq$(n-1,n)-$strong $\ddb-$manifold\rq\rq  is open but not closed.

They prove also in Theorem 4.9 (compare \cite{B} Theorem 1.1 for $pPL$ manifolds):

\lq\lq When $X_0$ is a $pK$ $\ddb-$manifold, then $X_t$ is $pK$ for small $t$\rq\rq.
\medskip

In \cite{RWZ2}, the authors get (Theorem 4.1):

\lq\lq If $X_0$ satisfies the  $(p,q+1)-$mild and $(q,p+1)-$mild $\ddb-$lemmata, then there is a $d-$closed $(p,q)-$form $\Omega(t)$ on  $X_t$ depending smoothly on $t$ with $\Omega(0) = \Omega_0$ for any $d-$closed $\Omega_0 \in {\E}^{p,q}(X_0)$\rq\rq.

Hence (Theorem 1.1):

\lq\lq When $X_0$ is a $pK$ $(p,p+1)-$mild $\ddb-$manifold, then $X_t$ is $pK$ for small $t$\rq\rq.

As for this last result, notice the connection with our Theorem 6.7: when $M$ is a $(p,p+1)-$mild $\ddb-$manifold, then it is also a $(p+1,p)-$weak $\ddb-$manifold,
so that a $pS$ structure (a real concept) gives rise to a $pWK$ structure.

\medskip

{\bf d) Deformation invariance of dimensions} 

In \cite{PU1},
the authors give informations about the dimensions of various cohomology groups with low indices,  as 
$h^{0,1}_{\db}, h^{0,1}_{\ddb}, h^{1,1}_{\ddb}$, referring to a comparison between $M$ and $M_t$. Moreover they prove that the property of being a sGG-manifold is open but it is not closed. Other interesting results, about $h^{0,q}_{\db}$ and $h^{p,0}_{\db}$, are given in \cite{ZR}, Theorem 1.9 and Theorem 1.10.

Also in \cite{RZ} the authors compare dimensions of cohomology groups (Theorem 3.1, Theorem 3.6, Theorem 3.7), in particular when $M$ is 
a $(p+1,q)-$mild $\ddb-$manifold and also a $\widetilde {(p,q+1)}-$mild $\ddb-$manifold (see Theorem 1.4 for the cases $p=0$ and $q=0$).

\end{document}